\documentclass[12pt,a4paper]{amsart}
\usepackage{amsfonts}
\usepackage{amsthm}
\usepackage{amsmath}
\usepackage{amscd}
\usepackage[latin2]{inputenc}
\usepackage{t1enc}
\usepackage[mathscr]{eucal}
\usepackage{indentfirst}
\usepackage{graphicx}
\usepackage{graphics}
\usepackage{pict2e}
\usepackage[margin=2.9cm]{geometry}
\usepackage{epstopdf}
\usepackage[colorlinks=true,urlcolor=green,linkcolor=red,citecolor=blue]{hyperref}

\theoremstyle{plain}
\newtheorem{Th}{Theorem}[section]
\newtheorem{Lemma}[Th]{Lemma}
\newtheorem{Cor}[Th]{Corollary}
\newtheorem{Prop}[Th]{Proposition}

 \theoremstyle{definition}
\newtheorem{Def}[Th]{Definition}

\newtheorem{?}[Th]{Problem}

\begin{document}

\title{ON THE  DISJOINT WEAK BANACH-SAKS OPERATORS}

\author{M. Berka$^{(1)}$}
\author{ O. Aboutafail$^{(1)}$}
\author{ J. H'michane$^{(2)}$}

\address{Address (1): Lab. Sciences et Ing\'{e}nierie. ENSA, Universit\'{e} Ibn Tofail 1400 K\'{e}nitra.}
\email{moulayothman.aboutafail@uit.ac.ma, mohamed.berka@uit.ac.ma}

\address{Address (2): Universit\'{e} Moulay Ismail, Facult\'{e} des Sciences, D\'{e}partement de Math\'{e}matiques, B.P. 11201 Zitoune, Mekn\`{e}s, Morocco.}
\email{hm1982jad@gmail.com}



\begin{abstract} We introduce and study a new class of operators that we call disjoint weak
Banach-Saks operators. We establish some characterizations of this class of operators by different types of convergence (norm convergence, unbounded order convergence, unbounded norm convergence and  unbounded absolute weak convergence) as well as by the positive weakly null sequences. Consequently, we give a new characterization of the disjoint weak Banach-Saks property by the positive disjoint weakly null sequences. Furthermore, we study the relationship between this class and other classes of operators.
\end{abstract}

\maketitle

\section{Introduction}
In  \cite{DOT}, Y.Deng, M.O'Brien and V.G.Troitsky introduced the disjoint weak Banach-Saks property (abb. DWBSP)  in the Banach lattice. Recall that a Banach lattice $E$ has the disjoint weak Banach-Saks property if, every disjoint weakly null  sequence in $E$ has a subsequence whose Ces\`{a}ro sequence is norm convergent in $E$, as examples of such Banach lattices we have $\ell_1$ and $L_p(c_0)$ (see \cite{DOT}).

In this paper we introduce the so called disjoint weak Banach-Saks operator from a Banach lattice $E$ into a Banach space $X$. Our definition is based on the disjoint weak Banach-Saks property. Mainly, in the Proposition \ref{caract_typeconv} we establish some characterizations of this class of operators by different types of convergence, and in the Proposition \ref{cor02} we give a characterization of this class of operators by positive weakly null sequences. Consequently, we give a new characterization of the disjoint weak Banach-Saks property. Also, we give a generalization of Proposition  6.9 \cite{GTX} and of Propositions 6.15 \cite{GTX} . Furthermore, we study the relationship between this class of operators  and that of weak Banach-Saks operators (resp; almost Banach-Saks operators, order weakly compact operators and weakly compact operators).
\section{Preliminaries and Notations}
To state our results, we need to fix some notation and recall some definitions. Let $E$
be a vector lattice. In a Riesz space, two elements $x$ and $y$ are said to be disjoint (in
symbols $x\perp y$) whenever $|x|\land |y| = 0$ holds. In a Banach lattice every order bounded disjoint sequence converges weakly to zero \cite{AB1}. For each $x, y \in E$ with $x \leq y$, the set $[x, y] = \{z \in E: x \leq z \leq y\}$ is called an order interval. A subset of $E$ is said to be order bounded if it is included in
some order interval. Recall that a nonzero element $x$ of a vector lattice $E$ is discrete
if the order ideal generated by $x$ equals the subspace generated by $x$. The vector
lattice $E$ is discrete, if it admits a complete disjoint system of discrete elements.
A Banach lattice is a Banach space $(E,||.||)$ such that $E$ is a vector lattice and
its norm satisfies the following property: for each $x, y\in E$ such that $|x| \leq |y|$, we
have $\|x\| \leq \|y\|$. If $E$ is a Banach lattice, its topological dual $E'$, endowed with the
dual norm, is also a Banach lattice.  A Banach lattice $E$ is order continuous, if for each generalized sequence $(x_{\alpha})$ such that $ x_{\alpha} \downarrow 0$ in $E$, the sequence  $(x_{\alpha})$ converges to $0$ for the norm $\|.\|$ where the notation $ x_{\alpha} \downarrow 0$  means that the sequence $(x_{\alpha})$ is decreasing, its infimum exists and $\inf(x_{\alpha}) = 0$.

 An order continuous Banach lattice $E$ is said to have the subsequence splitting property if for any norm bounded sequence $(x_n)$  there exist a subsequence $(x_{n_k})$ of $(x_n)$ and two sequences $(y_k)$ and $(z_k)$ such that $x_{n_k}=y_k+z_k$, $(y_k)$ is almost order bounded, $(z_k)$ is pairwise disjoint and $y_k \perp z_k$ for all $k$ (see \cite{GTX}).

 A vector lattice $E$ is said to be $\sigma-$laterally complete, if the supremum of every
disjoint sequence of $E^+$ exists in $E$.

 A Banach space is said to have the Schur property, whenever every weakly convergent sequence is norm convergent.

Recall from \cite{AB1} page 345 that a Banach lattice is said to have weakly sequentially continuous lattice operations whenever $x_n \overset{w}{\longrightarrow} 0$
implies  $|x_n| \overset{w}{\longrightarrow} 0$. Every$AM-$space has this property. Also,
any Banach lattice with the Schur property has weakly sequentially continuous lattice operations.

A net $(x_{\alpha})$ of a vector lattice $E$ is said to be uo-converge (abb. uo-convergent) to $x$ if $(|x_{\alpha}-x| \land u)$ converges in order to zero  for every $u\in E^+$; we write $x_{\alpha}\overset{uo}{\longrightarrow} x$. We mention that order convergence implies uo-convergence and they coincide for order bounded nets. We note that every disjoint net is uo-null (see \cite{GTX}).

 A net $(x_{\alpha})$ of a vector lattice $E$  is said to be an unbounded norm convergent (abb. un-convergent) to $x$ if $(|x_{\alpha}-x| \land u)$ converges to zero  for every $u\in E^+$; we write $x_{\alpha}\overset{un}{\longrightarrow} x$. The example 2.6 of \cite{DOT} shows that a disjoint sequence need not be un-null.

 A net $(x_{\alpha})$ of a vector lattice $E$  is said to be unbounded absolute weakly convergent (abb. uaw-convergent) to $x$ if $(|x_{\alpha}-x| \land u)$ converges to zero weakly for every $u\in E^+$; we write $x_{\alpha}\overset{uaw}{\longrightarrow} x$. Every absolute weakly convergent net is uaw-convergent. But the converse is not true in generale. The absolute weakly convergence coincides with uaw-convergence  for the order bounded nets (see \cite{OZ}). We note that every disjoint net is uaw-null (\cite[Lemma 2]{OZ}).

A Banach space $X$ has the weak Banach-Saks property if every weakly null sequence $(x_n)$ in $X$ has a subsequence $(x_{n_k})$  whose Ces\`{a}ro sequence $(\frac{1}{m}\sum^m_{k=1}x_{n_k})$ is norm convergent to zero.

 A subset $A$ of a Banach lattice $E$ is an \textbf{almost order bounded }if for any $\varepsilon > 0$ ther exists $u\in E^+$ such that $A\subset [-u,u]+\varepsilon B_E$. We know from \cite{GX2} that $A\subset [-u,u]+\varepsilon B_E$ iff $\sup_{x\in A}\|(|x|-u)^+\|\leq \varepsilon$ iff $\sup_{x\in A}\||x|-|u|\land |x|\|\leq \varepsilon$. Note that every norm convergent sequence  is almost order bounded.

We will use the term operator $T : E\longrightarrow F$ between two Banach lattices to mean
a bounded linear mapping. It is positive if $T (x) \geq 0$ in $F$ whenever $x \geq 0$ in $E$. The
operator $T$ is regular if $T = T_1-T_2$ where $T_1$ and $T_2$ are positive operators from $E$
into $F$. Note that each positive linear mapping on a Banach lattice is continuous.
If an operator  $T : E\longrightarrow F$  between two Banach lattices is positive, then its adjoint $T' : F'\longrightarrow E'$ is likewise positive, where $T'$ is defined by $T'(f)(x)=f(T(x))$ for each $f\in F'$ and for each $x\in E$.

  An operator $T:E \longrightarrow F$ between two Riesz spaces is said to preserve disjointness whenever $x\perp y$ in $E$ implies $Tx\perp Ty$  in $F$ (\cite[page 149]{AB1}).

A positive linear mapping $S$ is  disjointness preserving if, and only if, $S$ is a Riesz homomorphism (see \cite{Pagter}).

An operator $T:E\longrightarrow Y$  is called M-weakly compact, if $(T(x_n))$ is norm-null for
every bounded disjoint sequence $(x_n)$ in $E$ (see \cite{AB1}).

 An operator $T$ from a Banach lattice $E$ into a Banach space $X$ is said to be an almost disjoint  Banach-Saks, if for every bounded disjoint  sequence $(x_n)$ of $E$, $(T(x_n))$  has a Ces\`{a}ro convergent subsequence in $X$\cite{HHS}.

An operator $T: E\longrightarrow Y$ from a Banach lattice to a Banach space is said to be order weakly compact if, and on if, for every order bounded disjoint sequence $(x_n)$ of $E$, we have $\|T(x_n)\| \longrightarrow 0$ (\cite[Theorem 5.57 (Dodds)]{AB1}).

 An operator  $T$ is weakly compact if, and only if, for every norm bounded sequence $(x_n)$ of $X$ the sequence $(T(x_n))$ has a weakly convergent subsequence in $Y$ (see \cite{AB1}).

In the rest of this paper $X$, $Y$ will denote Banach spaces, and $E$, $F$ will denote Banach lattices.
\section{Main results}
We start by the following definition.
\begin{Def}\label{definition} An operator $T:E \longrightarrow Y$  is called disjoint weak Banach-Saks (abb; DWBS), if for each disjoint weakly null  sequence $(x_n)$ of $E$, the sequence $(T(x_n))$ has a subsequence whose Ces\`{a}ro sequence is norm convergent in $Y$.
\end{Def}
It is clear that $E$ has the disjoint weak Banach-Saks property if, and only if, the identity operator $Id_E$ of $E$ is disjoint weak Banach-Saks. And we observe that this class of operators contains thats of M-weakly compact operators.

We have the following characterisation of the positive disjoint weak Banach-Saks operator.
\begin{Th}\label{caract}
 Let $T:E \longrightarrow F$ be an operator. The following assertions are equivalent:
\begin{enumerate}
\item[(1)] $T$ is a positive disjoint weak Banach-Saks  operator.
\item[(2)] For each disjoint weakly null sequence $(x_n)\subset E^+$, the sequence $(T(x_n))$  has a subsequence whose Ces\`{a}ro sequence is norm convergent in $F$.
\end{enumerate}
\end{Th}
\begin{proof}
$(1)\Longrightarrow (2)$ Evident.\\
$(2)\Longrightarrow (1)$ Let $(x_n)$ be a disjoint weakly null sequence of $E$.  It follows from Remark 1 \cite{Wnuk1} that  $(|x_n|)$ is a disjoint sequence of $E^+$ such that $|x_n|\longrightarrow 0$ for the topologie $\sigma(E,E')$. From the assertion $(2)$, there exists $(|x_{n_k}|)$ a subsequence of $(|x_n|)$ such that $(T(|x_{n_k}|))$ is Ces\`{a}ro convergent. Now, by the inequality: $$|T(x_n)|\leq T(|x_n|)  \text{ for each n },$$ we obtain that $$|\frac{1}{m}\sum^m_{k=1}T(x_{n_k})|\leq \frac{1}{m}\sum^m_{k=1}|T(x_{n_k})|\leq \frac{1}{m}\sum^m_{k=1}T(|x_{n_k}|)$$
and since $\|\frac{1}{m}\sum^m_{k=1}T(|x_{n_k}|)\|\longrightarrow 0$, then $\|\frac{1}{m}\sum^m_{k=1}T(x_{n_k})\|\longrightarrow 0$. So, $(T(x_n))$ has a Ces\`{a}ro convergent subsequence in $F$ and hence $T$ is a positive disjoint weak Banach-Saks operator.
\end{proof}
As consequence of the last Theorem, we have the following characterization of the disjoint weak Banach-Saks property.
\begin{Cor}\label{caract00}
 The following assertions are equivalent:
\begin{enumerate}
\item[(1)] $E$ has the disjoint weak Banach-Saks property.
\item[(2)] Every disjoint weakly null sequence $(x_n)\subset E^+$  has a subsequence whose Ces\`{a}ro sequence is norm convergent.
\end{enumerate}
\end{Cor}
If $E$ is order continuous, we can obtain the following characterizations of positive disjoint weak Banach-Saks  operator.
\begin{Prop}\label{caract0}
 Let $T:E \longrightarrow F$ be an operator  such that $E$ is order continuous. The following assertions are equivalent:
\begin{enumerate}
\item[(1)] $T$ is a positive disjoint weak Banach-Saks  operator.
\item[(2)] For each  weakly null and uo-null sequence $(x_n)\subset E^+$, the sequence $(T(x_n))$  has a subsequence whose Ces\`{a}ro sequence is norm convergent in $F$.
\item[(3)] For each  weakly null sequence $(x_n)\subset E^+$, the sequence $(T(x_n))$  has a subsequence whose Ces\`{a}ro sequence is norm convergent in $F$.
\end{enumerate}
\end{Prop}
\begin{proof}

$(2)\Longrightarrow (1)$ It  follows from the fact that every disjoint sequence is uo-null.

$(1)\Longrightarrow (2)$ Let $(x_n)$ be a weakly null and uo-null sequence of $ E^+$. Since $E$ is order continuous, then by  Lemma 6.7 \cite{GTX} there exist a subsequence $(x_{n_k})$ of $(x_n)$ and a disjoint sequence $(d_k)$ of $E$ such that $\|x_{n_k}-d_k\| \longrightarrow 0$.  Since $(x_n)$ is a weakly null sequence of $E$, then it follows from the proof of Theorem 3.2 \cite{DOT} that $(d_k)$ is a weakly null sequence of $E$ and hence every subsequence of $(d_k)$ is  disjoint and weakly null. The assertion $(1)$ yields that  $(T(d_k))$  has a subsequence whose Ces\`{a}ro sequence is norm convergent in $F$.

$(3)\Longrightarrow (2)$ Obvious.

$(2)\Longrightarrow (3)$ Let $(x_n)$ be a  weakly null sequence of $E^+$. Since $E$ is order continuous,   then it follows from  Proposition 4.5 \cite{GTX} and  Proposition 4.7 of \cite{GTX} that $x_{n_k} \overset{uo}{\longrightarrow}0$ for some subsequence $(x_{n_k})$. The assertion $(2)$ yields  $(T(x_{n_k}))$   has a subsequence whose Ces\`{a}ro sequence is norm convergent in $F$.
\end{proof}
\begin{Prop}\label{compose1}
We have the following assertions:
\begin{enumerate}
\item If $T:E \longrightarrow X$ is a disjoint weak Banach-Saks operator, then for each operator $S:X\longrightarrow Y$ the composed operator $S\circ T$ is  disjoint weak Banach-Saks.
\item If $T:F \longrightarrow Y$ is a disjoint weak Banach-Saks operator and $S:E\longrightarrow F$ is a disjointness preserving operator, then the composed operator $T \circ S$ is disjoint weak Banach-Saks.
\end{enumerate}
\end{Prop}
\begin{proof}
\begin{enumerate}
\item Let $(x_n)$ be a disjoint weakly null sequence of $E$. Since $T$ is disjoint weakly Banach-Saks operator, then there exists $(x_{n_k})$ a subsequence of $(x_n)$ such that the Ces\`{a}ro means of $(Tx_{n_k})$ is convergent to $0$. Since $$\|\frac{1}{m}\sum_{k=1}^mS \circ Tx_{n_k}\|\leq \|S\|\|\frac{1}{m}\sum_{k=1}^mTx_{n_k}\|,$$ then the Ces\`{a}ro sequence of $(SoT(x_{n_k}))$ is convergent to $0$. Hence, $SoT$ is a disjoint weak Banach-Saks operator.
\item Let $(x_n)$ be a  disjoint weakly null sequence of $E$. Since $S$ is a disjointness preserving operator, then $(S(x_n))$ is a disjoint sequence and hence $(S(x_n))$ is a disjoint weakly null sequence in $F$. As $T$ is a disjoint weak Banach-Saks operator, then $(T\circ S(x_n))$ is a Ces\`{a}ro convergent subsequence in $Y$. Hence $T \circ S$ is a disjoint weak Banach-Saks operator.
\end{enumerate}
\end{proof}

As consequences of the above proposition, we have the following results:
\begin{Cor}\label{identity}
The following statements are equivalent:
\begin{enumerate}
\item For each  $Y$, every  operator $T:E \longrightarrow Y$ is disjoint weak Banach-Saks.
\item $E$ has the disjoint weak Banach-Saks property.
\end{enumerate}
\end{Cor}
\begin{Cor}\label{identity2}
The following statements are equivalent:
\begin{enumerate}
\item For each  $E$, every disjointness preserving operator $T:E \longrightarrow F$ is disjoint weak Banach-Saks.
\item $F$ has the disjoint weak Banach-Saks property.
\end{enumerate}
\end{Cor}
In the following result, we  prove that the class of  disjoint weak Banach-Saks operators satisfies the domination problem.
\begin{Th}\label{domination}
Let  $S,T: E\longrightarrow F$ be two operators with $0\leq S\leq T$. If $T$ is a disjoint weak Banach-Saks operator, then $S$ is also a disjoint weak Banach-Saks operator.
\end{Th}
\begin{proof}
Let  $S,T: E\longrightarrow F$ be two operators such that $0\leq S\leq T: E\longrightarrow F$. We suppose that $T$ is a disjoint weak Banach-Saks operator. Let $(x_n)$ be a disjoint weakly null sequence of $E$, it follows from Remark 1 \cite{Wnuk1} that $(|x_n|)$ is a disjoint weakly null sequence. Since $T$ is disjoint weak Banach-Saks, then there exists $(|x_{n_k}|)$ a subsequence of $(|x_n|)$ such that $(T(|x_{n_k}|))$ is Ces\`{a}ro convergent. On the other hand, we have $0\leq S\leq T$ implies that $|S(x_n)|\leq S(|x_n|)\leq T(|x_n|).$ This shows that $$\|\frac{1}{m}\sum^m_{k=1}S(x_{n_k})\|\leq \|\frac{1}{m}\sum^m_{k=1}|S(x_{n_k})|\|\leq \|\frac{1}{m}\sum^m_{k=1}T|x_{n_k}|\|$$ and hence  $(S(x_{n_k}))$ is Ces\`{a}ro convergent, that is $S$ is a disjoint weak Banach-Saks operator.
\end{proof}
We note that there exist Banach lattices $E$ and $ F$ and an operator  $T:E \longrightarrow F$ which is disjoint weak Banach-Saks and such that its modulus $|T |$ does not exist. We consider from \cite{ah} the operator $T:L_1[0,1] \longrightarrow c_0$  defined by:
$$\begin{array}{clcl}
  T: &L_1[0,1] &\longrightarrow c_0 &\\
      &f &\longmapsto &(\int^1_0f(x) r_1 dx, \int^1_0f(x) r_2 dx, ....)\\
 \end{array}$$
where $(r_n)$ is the sequence of Rademacher functions on $[0, 1]$. It is clear that $T$
is disjoint weak Banach-Saks ($L_1[0, 1]$ has the disjoint weak Banach-Saks property), but from the remark page 361 of \cite{ah} the modulus of $T$ does not exist.

\begin{Prop}\label{module}
Let $T:E \longrightarrow F$ be an order bounded disjointness preserving operator such that $F$ Dedekind-complet. Then, $T$ is disjoint weak Banach-Saks if, and only if, $|T|$ is disjoint weak Banach-Saks.
\end{Prop}
\begin{proof}
Let $T:E \longrightarrow F$ be an order bounded disjointness preserving operator.  By \cite[Theorem 2.2]{Boulabiar}, we have $$|T|(|x|)=|T(x)|=||T|(x)| \ \ \forall x\in E$$ and hence
$$\|\frac{1}{m} \sum^m_{k=1}|T|x_{n_k}\|=\||T|\frac{1}{m} \sum^m_{k=1}x_{n_k}\|=\||T\frac{1}{m} \sum^m_{k=1}x_{n_k}|\|=\|\frac{1}{m} \sum^m_{k=1}Tx_{n_k}\|,$$ as desired.
\end{proof}
\begin{Prop}\label{caract_typeconv}
If $E$ is order continuous, then for every operator $T : E\longrightarrow Y$  the following assertions are equivalent:
\begin{enumerate}
\item $T$ is disjoint weak Banach-Saks.
\item For each  weakly null sequence $(x_n)$ of $E$ such that $x_n \overset{un}{\longrightarrow}0$, the sequence $(T(x_n))$ has a subsequence whose Ces\`{a}ro sequence is norm convergent in $Y$.
\item For each weakly null sequence $(x_n)$ of $E$ such that $x_n \overset{uaw}{\longrightarrow}0$, the sequence $(T(x_n))$ has a subsequence whose Ces\`{a}ro sequence is norm convergent in $Y$.
\item For each weakly null sequence $(x_n)$ of $E$ such that $x_n \overset{uo}{\longrightarrow}0$, the sequence $(T(x_n))$ has a subsequence whose Ces\`{a}ro sequence is norm convergent in $Y$.
\end{enumerate}
\end{Prop}
\begin{proof}
$(1) \Longrightarrow (2)$ Let $(x_n)$ be a weakly null sequence of $E$ such that  $x_n \overset{un}{\longrightarrow}0$. It follows from Theorem 3.2 \cite{DOT} that there exist $(x_{n_k})$ a subsequence of $(x_n)$ and a disjoint sequence $(d_k)$ of $E$ such that $\|x_{n_k}- d_k \| \longrightarrow 0$. So,  $\|T(x_{n_k})- T(d_k) \| \longrightarrow 0$  in $Y$. Since $(x_n)$ is a weakly null sequence of $E$, then it follows from the proof of Theorem 3.2 \cite{DOT} that $(d_k)$ is a weakly null sequence of $E$. As $T$ is a disjoint weak Banach-Saks operator, the assumption of $(1)$ yields that  $(T(d_k))$  has a subsequence whose Ces\`{a}ro means are norm convergent in $Y$. Hence, $(T (x_{n_k}))$  has a subsequence whose Ces\`{a}ro sequence is norm convergent in $Y$.

$(2) \Longrightarrow (3)$ Let $(x_n)$ be a weakly null sequence of $E$ such that  $x_n \overset{uaw}{\longrightarrow}0$. Since $E$ is order continuous, then it follows from Theorem 4 \cite{OZ} that $x_n \overset{un}{\longrightarrow}0$. Therefore, the assertion (2) yields that $(T(x_n))$ has a subsequence whose Ces\`{a}ro sequence is norm convergent in $Y$.

$(3) \Longrightarrow (4)$ Let $(x_n)$ be a weakly null sequence of $E$ such that $x_n \overset{uo}{\longrightarrow}0$. Since $E$ is order continuous, then it follows from Propostion 2.5 \cite{DOT} and Theorem 4 \cite{OZ} that $x_n \overset{uaw}{\longrightarrow}0$. Hence, by the assertion (3) we infer that $(T(x_n))$ has a subsequence whose Ces\`{a}ro sequence is norm convergent in $Y$.

$(4) \Longrightarrow (1)$ Let $(x_n)$ be a disjoint weakly null sequence of $E$. By Corollary 3.6 \cite{GTX}, we see that $x_n \overset{uo}{\longrightarrow}0$ and hence it follows from the assertion (4) that $(T(x_n))$ has a Ces\`{a}ro convergent subsequence in $Y$. That is, $T$ is a disjoint weak Banach-Saks operator.
\end{proof}
\begin{Prop}\label{caract_almost ob}
If $F$ is order continuous, then for every  operator $T: E\longrightarrow F$  the following assertions are equivalent.
\begin{enumerate}
\item $T$ is disjoint weak Banach-Saks.
\item For each disjoint weakly null sequence $(x_n)$ of $E$, $(T(x_n))$ has a subsequence whose Ces\`{a}ro sequence is almost order bounded.
\item For each disjoint weakly null sequence $(x_n)$ of $E^+$, $(T(x_n))$ has a subsequence whose Ces\`{a}ro sequence is almost order bounded.
\end{enumerate}
\end{Prop}
\begin{proof}
$(1) \Longrightarrow (2) \Longrightarrow (3)$ Are obvious.

$(2) \Longrightarrow (1)$ Let $(x_n)$ be a  disjoint weakly null sequence of $E$. Since any subsequence
$(x_{n_k})$  of $(x_n)$ is also a  disjoint weakly null sequence of $E$, the assertion  $(2)$ yields that $(T(x_{n_k}))$ has a subsequence whose Ces\`{a}ro sequence is almost order bounded. In particular, any subsequence of $(T(x_n))$ has a further subsequence whose Ces\`{a}ro sequence is almost order bounded.  Since $F$ is order continuous, then it follows from  Lemma 6.3 of \cite{GTX} that there exists a subsequence $(T(x_{n_k}))$ of $(T(x_n))$ and a vector $y\in F$ such that the Ces\`{a}ro sequence is of any subsequence of $(T(x_{n_k}))$ converges uo- and in norm to $y=0$.

$(3) \Longrightarrow (2)$  Let $(x_n)$ be a  disjoint weakly null sequence of $E$. It follows from Wnuk  \cite[Remark 1]{Wnuk1} that  $(|x_n|)$ is a disjoint sequence of $E^+$ such that $|x_n|{\overset{w}{\longrightarrow}} 0$. We have $x_n=x_n^+-x_n^-$ and $|x_n|=x_n^++x_n^-$. It is easy  from the inequality $0\leq x_n^+ \leq |x_n|$ (resp,  $0\leq x_n^- \leq |x_n|$), to see that  $(x_n^+)$ (resp, $(x_n^-)$) is a disjoint weakly null sequence of $E^+$. So, every subsequence of $(x_n^+)$ is disjoint weakly null in $E^+$. The assumption $(3)$ yields that every subsequence of $(T(x_n^+))$ has a further subsequence whose Ces\`{a}ro sequence is almost order bounded in $F$. Since $F$ is order continuous, then by  Lemma 6.3 \cite{GTX} there exists a subsequence $(T(x_{\phi(n)}^+))$ of $(T(x_n^+))$ and a vector $y\in F$ such that the Ces\`{a}ro sequence  of any subsequence of $(T(x_{\phi(n)}^+))$ converges uo and in norm to $y=0$, and so the Ces\`{a}ro sequence  of any subsequence of $(T(x_{\phi(n)}^+))$ is almost order bounded in $F$. On the other hand, $(x_{\phi(n)}^-)$ is a disjoint  weakly null sequence of $E^+$. It follows from the assertion $(3)$ that the sequence $(T(x_{\phi(n)}^-))$ has a subsequence $(T(x_{\psi(\phi(n))}^-))$  whose Ces\`{a}ro sequence  is almost order bounded in $F$. As it was shown above, the Ces\`{a}ro sequence  of any subsequence of $(T(x_{\phi(n)}^+))$ is almost order bounded in $F$ and hence the Ces\`{a}ro sequence  of $(T(x_{\psi(\phi(n))}^+))$ is almost order bounded. Therefor, the Ces\`{a}ro sequence of the subsequence  $T(x_{\psi(\phi(n))})=T(x_{\psi(\phi(n))}^+)-T(x_{\psi(\phi(n))}^-)$ of the sequence $(T(x_n))$ is almost order bounded in $F$, where $\phi$ and $\psi$ are increasing mappings from $\mathbb{N}$ to  $\mathbb{N}$.
\end{proof}
Whenever $E$ and $F $ are order continuous, we obtain the following characterizations.
\begin{Prop}\label{cor02}
Let $E$  and $F$ be order continuous. For an operator $T: E\longrightarrow F$, the following assertions are equivalent.
\begin{enumerate}
\item $T$ is  disjoint weak Banach-Saks.
\item For each weakly and uo-null sequence $(x_n)$ of $E^+$, $(T(x_n))$ has a subsequence whose Ces\`{a}ro sequence is almost order bounded in $F$.
\item For each weakly null sequence $(x_n)$ of $E^+$, $(T(x_n))$ has a subsequence whose Ces\`{a}ro sequence is almost order bounded in $F$.
\end{enumerate}
\end{Prop}
\begin{proof}
$(1) \Longrightarrow(2) \Longrightarrow (3)$ Are obvious.

$(2) \Longrightarrow(1)$ Let $(x_n)$ be a weakly null and uo-null sequence of $E^+$. Since $E$ is order continuous, then by  Lemma 6.7 \cite{GTX} there exist a subsequence $(x_{n_k})$ of $(x_n)$ and a disjoint sequence $(d_k)$ of $E$ such that $\|x_{n_k}-d_k\| \longrightarrow 0$.  Since $(x_n)$ is a weakly null sequence of $E$, then it follows from the proof of Theorem 3.2 \cite{DOT} that $(d_k)$ is a weakly null sequence of $E$ and hence every subsequence of $(d_k)$ is  disjoint and weakly null. As $T$ is disjoint weak Banach-Saks, then it follows from the Proposition \ref{caract_almost ob} that every subsequence of $(T(d_k))$  has a further subsequence whose Ces\`{a}ro sequence is almost order bounded in $F$. As $F$ is order continuous, then it follows from Lemma 6.3 \cite{GTX} that $(T(d_k))$ has a subsequence whose Ces\`{a}ro sequence is norm convergent in $F$ and hence $(T (x_{n_k}))$  has a subsequence whose Ces\`{a}ro sequence is almost order bounded in $F$.

$(3) \Longrightarrow(2)$ Let $(x_n)$ be a weakly null  sequence of $E^+$. Since $E$ is order continuous, then it follows from the Proposition 4.5 and the Proposition 4.7 \cite{GTX} that there exists a subsequence $(x_{n_k})$ of $(x_n)$ such that $x_{n_k}\overset{uo}{\longrightarrow}0$. The assertion $(3)$ yields that $(T(x_{n_k}))$  has a subsequence whose Ces\`{a}ro sequence is almost order bounded in $F$.
\end{proof}
As consequence of the Proposition \ref{caract0}, the Proposition \ref{caract_almost ob} and the Proposition \ref{cor02}, we get the Proposition 6.9 \cite{GTX}.
\begin{Cor}(Proposition 6.9 \cite{GTX})\\
Let $E$ be  order continuous. Then, the following assertions are equivalent.
\begin{enumerate}
\item $E$ has the disjoint weak Banach-Saks property.
\item Every disjoint weakly null sequence $(x_n)$ of $E$ has a subsequence whose  Ces\`{a}ro sequence is almost order bounded.
\item Every disjoint weakly null  sequence $(x_n)$ of $E^+$ has a subsequence whose  Ces\`{a}ro sequence is almost order bounded.
\item Every weakly and uo-null sequence $(x_n)$ of $E$ has a subsequence whose Ces\`{a}ro sequence is norm convergent.
\item Every weakly and uo-null sequence $(x_n)$ of $E^+$ has a subsequence whose  Ces\`{a}ro sequence is almost order bounded.
\item Every weakly null sequence $(x_n)$ of $E^+$ has a subsequence whose Ces\`{a}ro sequence is norm convergent.
\item Every weakly null sequence $(x_n)$ of $E^+$ has a subsequence whose Ces\`{a}ro sequence is almost order bounded.
\end{enumerate}
\end{Cor}
We denote by $L_{DWBS}(E,Y)$ the spaces of all disjoint weak Banach-Saks operators from $E$ to $Y$.
\begin{Prop}
 $L_{DWBS}(E,Y)$ is a closed subset of the space of all continuous operators from  $E$ to $F$.
\end{Prop}
\begin{proof}
Let $(T_{m})$ be a sequence of disjoint weak Banach-Saks operators which is convergent to the operator $T$. We will show that $T$ is also a disjoint weak Banach-Saks operator. For this, let $(x_{n})$ be a disjoint  weakly null sequence of $E$, then  $(T_m(x_n))$  has a subsequence $(T_m(x_{n_k}))$  whose Ces\`{a}ro sequence is norm convergent in $Y$. That is, $\lim \frac{1}{n} \sum ^{n}_{k=1}T_m(x_{n_k})=0$. On the other hand,  $(T_{m})$ converges to the operator $T$. So, given any $\varepsilon>0$, there is an $m_{0} \in \mathbb{N}$ such that $\|T_{m}-T\|\leq \frac{\varepsilon}{2M}$ for each $m>m_{0}$, where $M=\sup_{n}\|x_{n} \|$. Let $m>m_{0}$, for sufficiently large $n$ we have $\|\frac{1}{n} \sum ^{n}_{k=1}T_m(x_{n_k})|\ \leq\frac{\varepsilon}{2}$.\\
  Therefore, by the inequality
 \begin{eqnarray*}
   \|T(x_{n})\| &\leq& \|(T-T_{m})(x_{n})\|+\|T_{m}(x_{n})\| \\
     &\leq&  M. \frac{\varepsilon}{2M}+\|T_{m}(x_{n})\|
 \end{eqnarray*}
 we obtain that
  \begin{eqnarray*}
   \|\frac{1}{n} \sum ^{n}_{k=1}T(x_{n_k})\| &\leq&  \frac{\varepsilon}{2}+\|\frac{1}{n} \sum ^{n}_{k=1}T_m(x_{n_k})\| \\
     &\leq&  \frac{\varepsilon}{2}+\frac{\varepsilon}{2}=\varepsilon.
 \end{eqnarray*}
Hence, $\lim \frac{1}{n} \sum ^{n}_{k=1}T(x_{n_k})=0$, that is $T \in L_{DWBS}(E,Y)$, as desired.
\end{proof}
\begin{Prop}\label{svect}
 If $F$ is order continuous, then $L_{DWBS}(E,F)$ is a closed vector subspace of the space of all continuous operators from $E$ to  $F$.
\end{Prop}
\begin{proof}
The proof is the same as that of the implication $(3) \Longrightarrow(2)$ of Proposition \ref{caract_almost ob}.
 \end{proof}
As consequence of the proposition \ref{svect}, we have the following result.
\begin{Cor}
Let $T: E \longrightarrow F$ be an operator such that $F$ is order continuous.  If the modulus of $T$ exists and is disjoint weak Banach-Saks, then $T$  is disjoint weak Banach-Saks.
\end{Cor}
\begin{proof}
The proof  is the same as that of the seconde part of  Proposition 2.1 of \cite{ah}.
\end{proof}
\begin{Lemma}\label{uaw-lem}
Let $(x_{\alpha})$ be a net in  $E$ such that $x_{\alpha}\overset{uaw}{\longrightarrow}0$. Then, there exists an increasing sequence of indices $(\alpha_k)$ and a disjoint sequence
$(d_k)$ such that $x_{\alpha_k}-d_k \overset{w}{\longrightarrow}0$.
\end{Lemma}
\begin{proof}
The proof of this lemma is the same as that of Theorem 3.2 \cite{DOT}. It suffice to replace the norm  in the inequality $\|x_{\alpha_k}\land x_{\alpha_i}\|\leq \frac{1}{2^{k+1}} $ by a positive linear functional in order to replace the  norm convergence by the weak convergence.
\end{proof}
\begin{Prop}\label{uaw-dwbs0}
If  $Y$ has the Schur property, then for every  operator $T : E\longrightarrow Y$ the following statements are equivalent:
\begin{enumerate}
\item $T$ is  disjoint weak Banach-Saks.
\item For each weakly null  sequence $(x_n)$ of $E$ such that $x_n \overset{uaw}{\longrightarrow}0$, $(T(x_n))$ has a subsequence whose Ces\`{a}ro sequence is norm convergent.
\end{enumerate}
\end{Prop}
\begin{proof}
$(1) \Longrightarrow (2)$ Let $(x_n)$ be a weakly null sequence of $E$ such that  $x_n \overset{uaw}{\longrightarrow}0$. It follows from the Lemma \ref{uaw-lem} that there exist $(x_{n_k})$ a subsequence of $(x_n)$ and a disjoint sequence $(d_k)$ of $E$ such that $x_{n_k}- d_k \overset{w}{\longrightarrow} 0$. As $(x_{n_k})$ is weakly null in $E$, then $(d_k)$ is weakly null in $E$. Since $T$ is a disjoint weak Banach-Saks operator, then  $(T(d_k))$ has a Ces\`{a}ro norm convergent subsequence $(T(d_{\phi(k)}))$ which is also weakly convergent. As $x_{n_{\phi(k)}}- d_{\phi(k)} \overset{w}{\longrightarrow} 0$, then $(T (x_{n_k}))$ has a Ces\`{a}ro weakly convergent subsequence $(T(x_{n_{\phi(k)}}))$ in $Y$ and since $Y$ has the Schur property, then $(T (x_{n_k}))$ has a Ces\`{a}ro  convergent subsequence $(T(x_{n_{\phi(k)}}))$.

$(2) \Longrightarrow (1)$ Let $(x_n)$ be a disjoint weakly null sequence of $E$. It follows from Lemma 2 \cite{OZ} that $x_n \overset{uaw}{\longrightarrow}0$ and hence  $(T(x_n))$ has a Ces\`{a}ro convergent subsequence in $Y$.
\end{proof}
As a consequence of the Proposition \ref{caract_almost ob} and the Proposition \ref{uaw-dwbs0}, we obtain the following result.
\begin{Cor}\label{uaw-dwbs}
If the lattice operations of $E$ are  weakly sequentially continuous and $Y$ has the Schur property, then for every operator $T : E\longrightarrow Y$ the following  statements are equivalent:
\begin{enumerate}
\item $T$ is  disjoint weak Banach-Saks.
\item For each weakly null sequence $(x_n)$ of $E$ such that $x_n \overset{uaw}{\longrightarrow}0$, $(T(x_n))$ has a subsequence whose Ces\`{a}ro sequence is norm convergent.
\item For each weakly null  sequence $(x_n)$ of $E$, $(T(x_n))$ has a subsequence whose Ces\`{a}ro sequence is norm convergent.
\item For each weakly null  sequence $(x_n)$ of $E^+$, $(T(x_n))$ has a subsequence whose Ces\`{a}ro sequence is norm convergent.
\end{enumerate}
\end{Cor}
\begin{proof}
$(1) \Longleftrightarrow (2)$ Proposition \ref{uaw-dwbs0}.

$(3) \Longrightarrow (2)$ Its obvious.

$(2) \Longrightarrow (3)$ Let $(x_n)$ be a weakly null sequence of $E$. Since  the lattice operations of $E$ are weak sequentially continuous, then it follows from Theorem 2.1 \cite{elb} that $x_n \overset{uaw}{\longrightarrow}0$. Hence, the statement $(2)$ yields that $(T(x_n))$ has a subsequence whose Ces\`{a}ro means is norm convergent in  $Y$.

$(3) \Longrightarrow (4)$ Its obvious.

$(4) \Longrightarrow (3)$ The proof is the same as of the implication $(3) \Longrightarrow(2)$ of Proposition \ref{caract_almost ob}.
\end{proof}
\begin{Cor}\label{uaw-dwbs00}
 Let $T:E \longrightarrow F$ be an operator such that $F$ has the Schur property. The following assertions are equivalent:
\begin{enumerate}
\item $T$ is a positive disjoint weak Banach-Saks operator.
\item For each weakly null  sequence $(x_n)$ of $E^+$ such that $x_n \overset{uaw}{\longrightarrow}0$, $(T(x_n))$ has a subsequence whose Ces\`{a}ro sequence is norm convergent in  $F$.
\item For each  weakly null sequence $(x_n)\subset E^+$, the sequence $(T(x_n))$  has a subsequence whose Ces\`{a}ro sequence is norm convergent in $F$.
\end{enumerate}
\end{Cor}
\begin{proof}
$(3) \Longrightarrow (2)$  Obvious.

 $(3) \Longrightarrow (1)$ It follows from the Theorem \ref{caract}.

$(2) \Longrightarrow (3)$ Let $(x_n)$ be a  weakly null sequence of $E^+$, then  $x_n \longrightarrow 0$ in the absolute weak topology and hence it follows  from \cite{OZ} that  $x_n \overset{uaw}{\longrightarrow}0$. By the assumption $(2)$, the sequence $(T(x_n))$  has a subsequence whose Ces\`{a}ro sequence is norm convergent in $F$.

$(2) \Longrightarrow (1)$ Let $(x_n)$ be a disjoint weakly null sequence of $E^+$. It follows from Lemma 2 \cite{OZ} that $x_n \overset{uaw}{\longrightarrow}0$ and hence  $(T(x_n))$ has a Ces\`{a}ro convergent subsequence in $F$.

$(1) \Longrightarrow (2)$  The proof is the same as of the implication $(1) \Longrightarrow(2)$ of Proposition \ref{uaw-dwbs0}.
\end{proof}
Note that a weak Banach-Saks operator is disjoint weak Banach-Saks. But the converse is not true in general. In fact,  the Banach lattice $E=L_p(c_0)=L_p([0,1];c_0)$, where $0<p<1$,  has the disjoint weak Banach-Saks property (\cite[example 6.10]{GTX}) but fails to have the weak Banach-Saks property. So, $Id_{E}$ the identity operator of $E$ is a disjoint weak Banach-Saks operator which is not weak Banach-Saks.

In the following result, we give sufficient conditions on  $E$  and   $Y$ under which  each  disjoint weak Banach-Saks operator from $E$ to $Y$ is  weak Banach-Saks, for each $Y$.
\begin{Th}\label{dwbs-wbs}
 Each disjoint weak Banach-Saks operator $T:E \longrightarrow Y$  is weak Banach-Saks, if one of the following assertions is valid:
\begin{enumerate}
\item $E$ has the subsequence splitting property.
\item $E$ is order continuous and atomic.
\item $E$ and $E^{\prime}$ are order continuous.
\item The lattice operations of $E$ are weakly sequentially continuous and $Y$ has the Schur property.
\end{enumerate}
\end{Th}
\begin{proof}
$1)$  Let $(x_n)$ be a weakly null sequence of $E$. Since $E$ has the subsequence splitting property, by passing to a subsequence we may assume that $x_n=y_n+z_n$, where $(y_n)$ is almost order bounded and  $(z_n)$ is disjoint. Since $E$ is order continuous, then by passing to a further subsequence, we may assume that every subsequence of $(y_n)$ and therefore of $(T(y_n))$ is Ces\`{a}ro convergent (\cite[lemma 6.3]{GTX}). We have $(x_n)$ is a weakly null sequence, then $(z_n)$ is also a weakly null sequence. Since $T$ is a disjoint weak Banach-Saks operator, then $(T(z_n))$ has a Ces\`{a}ro convergent subsequence $(T(z_{n_k}))$. We put $T(x_{n_k})=T(y_{n_k})+T(z_{n_k})=T(y_{n_k}+z_{n_k})$, where $(y_{n_k})$ is a subsequence of $(y_n)$. We have $(T(x_{n_k}))$ is a Ces\`{a}ro convergent subsequence of $(T(x_n))$, hence $T$ is a  weak Banach-Saks operator.

$2)$ It follows from Lemma 6.7 and Lemma 6.14 \cite{GTX}.

$3)$ Let $(x_n)$ be a weakly null sequence of $E$. Since  $E$ and $E^{\prime}$ are order continuous, then it follows from Theorem 5 \cite{w1}  that $(x_n)$ is an uo-null  sequence of $E$. As $T$ is a disjoint weak Banach-Saks operator, then by the Proposition \ref{caract_typeconv} we infer that $(T(x_n))$ has a subsequence whose Ces\`{a}ro sequence are norm convergent in $Y$. That is $T$ is weak Banach-Saks.

$4)$ It follows from the Corollary \ref{uaw-dwbs}.
\end{proof}
Note that an almost Banach-Saks operator is disjoint weak Banach-Saks. But the converse is not true in general. In fact, the identity operatot $Id_{\ell^1}: \ell^1 \longrightarrow \ell^1$ is disjoint weak Banach-Saks (because $\ell^1$ has the disjoint weak Banach-Saks property), but follows from \cite{HHS} that $Id_{\ell^1}$  is not almost Banach-Saks.

 In the following result, we give necessary and sufficient conditions on  $E$ and  $Y$ under which each disjoint weak Banach-Saks operator from $E$ to $Y$ is almost Banach-Saks.
\begin{Th}\label{dwbs-abs}
 The following assertions are equivalent:
 \begin{enumerate}
 \item  Each disjoint weak Banach-Saks operator  $T:E\longrightarrow Y$ is almost Banach-Saks.
 \item  One of the following assertions is valid:
\begin{enumerate}
\item[(a)]  $E^{\prime}$ is order continuous,
\item[(b)]  $Y$  has the  Banach-Saks property.
\end{enumerate}
\end{enumerate}
\end{Th}
\begin{proof}
$(1) \Longrightarrow (2)$
 Assume that neither the norm of $E^{\prime}$ is order continuous nor $Y$ has the  Banach-Saks property. Then, by Theorem 2.4.14 and Proposition 2.3.11 \cite{MN} $E$ contains a complemented copy of $\ell^{1}$ and there exists a positive projection $P:E\longrightarrow \ell^1$, on the other hand since $Y$ does not have the  Banach-Saks property then, there exists $(y_n)$ a norm bounded sequence of $Y$ with no Ces\`{a}ro convergent subsequences.

We consider the following operator:
$$\begin{array}{clcl}
  S: &\ell^{1}  &\longrightarrow Y&\\
      &(\lambda_n) &\longmapsto &\sum_{n=1}^{\infty}\lambda_n y_n\\
 \end{array}$$
$S$ is well defined.

Now, we consider the composed operator $T=S\circ P$. Since $\ell^1$ has the weak Banach-Saks property, then both $S$ and $P$ are weak Banach-Saks operators. It follows that $T$ is weak Banch-Saks and hence $T$ is disjoint weak Banach-Saks. To finish the proof, we have to claim that $T$ is not an almost Banach-Saks operator. Otherwise, since the injection $\imath: \ell^1\longrightarrow E $ is a lattice homomorphism, then it follows from  Propostion 3.2  \cite{HHS} that $T \circ \imath  $ is an almost Banach-Saks operator. But by taking $(e_n)$ the unit basis of $\ell^1$ as a norm bounded disjoint sequence, we have $T \circ \imath(e_n)=y_n$ with no Ces\`{a}ro convergent subsequence, which is a contradiction.

$(2)(a) \Longrightarrow (1)$ Let $T:E\longrightarrow Y$ be a disjoint weak Banach-Saks operator and let $(x_n)$ be a norm bounded disjoint sequence in $E$. Since $E'$ is order continuous, then by \cite[theorem 2.4.14]{MN} the sequence $(x_n)$ is also weakly null and so  $(Tx_n)$ has a subsequence whose Ces\`{a}ro sequence is norm convergent in $Y$. Hence, the operator $T$ is  almost Banach-Saks.

$(2)(b) \Longrightarrow (1)$  Its obvious.
\end{proof}
As  consequences of the Theorem \ref{dwbs-abs}, we have the following results.
\begin{Cor}\label{cor03}
The following assertions are equivalent:
 \begin{enumerate}
 \item  Each disjoint weak Banach-Saks operator  $T:E\longrightarrow E$ is almost Banach-Saks,
 \item $E^{\prime}$ is order continuous.
\end{enumerate}
\end{Cor}
\begin{Cor}(Proposition 6.15 \cite{GTX})
A Banach lattice $E$ with the DWBSP has the DBSP if and only if it contains no lattice copy of $\ell^1$.
\end{Cor}
Recall from   Theorem 5.57 \cite{AB1} that an operator $T: E\longrightarrow Y$ is order weakly compact if, and on if,  for evry order bounded disjoint sequence $(x_n)$ of $E$,  we have  $\|T(x_n)\| \longrightarrow 0$.
\begin{Th}\label{owc_dwbs}
 If either $E$ is $\sigma$-laterally complete or AM-space with unit, then every o-weakly compact operator from $E$ to $Y$ is disjoint weak Banach-Saks.
\end{Th}
\begin{proof}
 Let $T: E\longrightarrow Y$ be an order weakly compact operator and let $(x_n)$ be a disjoint weakly sequence in $E$. Since $E$ is $\sigma$-laterally complete (resp, $E$ is AM-space with unit), then $(x_n)$ is order bounded and hence $(T(x_n)$ is norm-null. So, $(T(x_n))$ has a Ces\`{a}ro convergent subsequence.
\end{proof}
We note that a disjoint weak Banach-Saks operator is not necessary o-weakly compact.  In fact, the identity operator $Id_{\ell^{\infty}}: \ell^{\infty} \longrightarrow \ell^{\infty}$ is disjoint weak Banach-Saks (because $\ell^{\infty}$ has the disjoint weak Banach-Saks property) but it  is not o-weakly compact(because $\ell^{\infty}$ is not order continuous).

By the proof of Theorem 2.2 \cite{AH2}, we can investegate the following result.
\begin{Th}
 If $E$ has the disjoint weak Banach-Saks property, then the following assertions are equivalent:
 \begin{enumerate}
 \item  Each order bounded operator from $E$ to $F$ is order weakly compact.
 \item  Each order bounded  disjoint weak Banach-Saks operator  $T:E\longrightarrow F$ is order weakly compact.
 \item  One of the following assertions is valid:
\begin{enumerate}
\item[(a)]  $E$ is order continuous,
\item[(b)]  $F$ is order continuous.
\end{enumerate}
\end{enumerate}
\end{Th}
We note that the identity operator $Id_{c_0}$  of the Banach lattice $c_0$ is a disjoint weak Banach-Saks operator which is not weakly compact. Conversely, weakly compact operators are not in general  disjoint weak Banach-Saks. In fact, the identity operator of the Baerstein space (see \cite{B}) is weakly compact but fails to be disjoint weak Banach-Saks.
\begin{Th}\label{wc_dwbs}
If either $E^{\prime}$  has the positive Schur property or $Y$ has the schur property, then every weakly compact operator from $E$ to $Y$ is  disjoint weak Banach-Saks.
\end{Th}
\begin{proof}
\begin{itemize}
\item  Let $T: E\longrightarrow Y$ be a weakly compact operator. If $E^{\prime}$  has the positive Schur property, then it follows from  Proposition 3.18 \cite{HHS} that $T$ is  almost Banach-Saks and hence $T$ is  disjoint weak Banach-Saks.
\item Let $(x_n)$ be a weakly disjoint sequence in $E$ and $T: E\longrightarrow Y$ be a weakly compact operator. It follows from \cite{AB1} that  $T(x_n)$ has a weakly convergent subsequence in $Y$. Since $Y$ has the Schur property, then $(T(x_n))$ has a norm convergent subsequence.
\end{itemize}
\end{proof}

\end{document}